\documentclass[12pt,reqno]{amsart}
\usepackage[T1]{fontenc}
\usepackage[latin9]{inputenc}
\usepackage{amsthm,url}
\usepackage{amssymb}

\usepackage[a4paper]{anysize}\marginsize{3.5cm}{2.5cm}{2.5cm}{3.5cm}

\usepackage{graphicx}

\usepackage{color}
\RequirePackage{amsthm}
\RequirePackage{amssymb}
\usepackage[all,ps]{xy}
\usepackage{latexsym}

\def\Z{{\mathbb Z}}

\def\Q{{\mathbb Q}}
\def\C{{\mathbb C}}

\def\card{\mathop{\rm Card}\nolimits}
\def\Cl{\mathop{\rm Cl}\nolimits}
\def\codim{\mathop{\rm codim}\nolimits} 

\def\Div{\mathop{\rm Div}\nolimits}

\def\Pic{\mathop{\rm Pic}\nolimits}
\def\supp{\mathop{\rm Supp}\nolimits}

\def\tilde{\widetilde}

\def\phi{\varphi}

\def\Pic{\mathop{\rm Pic}\nolimits}
\def\dim{\mathop{\rm dim}\nolimits}
\def\div{\mathop{\rm div}\nolimits}

\def\Card{\mathop{\rm Card}\nolimits}

\newtheorem{thm}{Theorem}[section]

\newtheorem{cor}[thm]{Corollary}
\newtheorem{definition}[thm]{Definition}

\newtheorem{prop}[thm]{Proposition}

\theoremstyle{plain}

\newtheorem{lemma}[thm]{Lemma}
\newtheorem{corollary}[thm]{Corollary}

\theoremstyle{definition}
\newtheorem{remark}[thm]{Remark}

\begin{document}

\title{On the Boucksom-Zariski decomposition for irreducible symplectic varieties and bounded negativity}
\author{Micha\l\ Kapustka, Giovanni Mongardi, Gianluca Pacienza, Piotr Pokora}
\date{\today}



\address{Micha\l\ Kapustka \newline
	Institute of Mathematics of the Polish Academy of Sciences,
	ul. \'{S}niadeckich 8,
	00-656 Warszawa, Poland.}
\email{michal.kapustka@impan.pl}

\address{Giovanni Mongardi \newline
	Dipartimento di Matematica, Universit\`a degli studi di Bologna, Piazza Di Porta San Donato 5, Bologna, Italia 40126.}
\email{giovanni.mongardi2@unibo.it}

\address{Gianluca Pacienza \newline
	Institut Elie Cartan de Lorraine, Universit\'e de Lorraine et CNRS, F-54000 Nancy, France.}
\email{gianluca.pacienza@univ-lorraine.fr}

\address{Piotr Pokora \newline
Department of Mathematics,
University of the National Education Commission Krakow,
Podchor\c a\.zych 2,
PL-30-084 Krak\'ow, Poland.}
\email{piotr.pokora@uken.krakow.pl, piotrpkr@gmail.com}

\begin{abstract}
	Zariski decomposition plays an important role in the theory of algebraic surfaces due to many applications. For irreducible symplectic manifolds Boucksom provided a characterization of his divisorial Zariski decomposition in terms
	of the Beauville-Bogomolov-Fujiki quadratic form.  Different variants of singular holomorphic symplectic varieties have been extensively studied in recent years. 
	In this note we first show that the ``Boucksom-Zariski'' decomposition holds for effective divisors in the largest possible framework of varieties with symplectic singularities. 
	On the other hand in the case of projective surfaces, it was recently shown that there is a strict relation between the boundedness of coefficients of Zariski decompositions of pseudoeffective integral divisors and the bounded negativity conjecture. In the present note, we show that an analogous phenomenon can be observed in the case of  projective  irreducible symplectic varieties. We furthermore prove an effective analog of the bounded negativity conjecture in the smooth case. Combining these results we obtain information on the denominators of ``Boucksom-Zariski'' decompositions for holomorphic symplectic manifolds. From such a bound we easily deduce a result of effective birationality for big line bundles on projective holomorphic symplectic manifolds, answering to a question asked by F. Charles.
\end{abstract}

\maketitle

\section{Introduction}
Zariski decomposition is a fundamental tool in the theory of linear series on algebraic surfaces. 
One possible generalization in higher dimension is the divisorial Zariski decomposition (see \cite{Nak} in the projective case and \cite{Boucksom1} in the K\"ahler setting).
When $X$ is an irreducible symplectic manifold, by replacing the intersection form with the Beauville-Bogomolov-Fujiki quadratic form $q_X$ on $H^2(X,\mathbb C)$,
Boucksom  \cite[Theorem 4.3 part (i), Proposition 4.4, Theorem 4.8 and Corollary 4.11]{Boucksom1} gave a characterization of his divisorial Zariski decomposition in terms of the quadratic form $q_X$ (see Definition \ref{definition:q-Zar}).  We will refer to it as a {\it Boucksom-Zariski decomposition}.  
Several notions of singular irreducible symplectic varieties have received much attention in recent years, for several reasons. Let us simply mention here that the minimal model program naturally leads to singular minimal models. ``Singular'' irreducible symplectic varieties are studied in the theory of orbifolds or V-manifolds.  They are also studied as moduli spaces of sheaves on $K3$ or abelian surfaces (see \cite{PR} and references to prior works therein). The period map and the moduli theory of ``singular'' irreducible symplectic varieties are extensively studied in \cite{BL1, BL2}.  Finally ``singular'' irreducible symplectic varieties appear as building blocks of mildly singular projective varieties with trivial canonical class (see \cite{HP, Dru18, GGK}).  See the recent survey \cite{Perego} for all the different definitions and results. 
As the Boucksom-Zariski decomposition proved to be a very useful tool in the theory of smooth irreducible symplectic varieties  it is natural to ask whether it holds for some classes
of singular symplectic varieties.
In this article we show  that the (Boucksom characterization of the divisorial) Zariski decomposition actually holds in the largest possible framework, namely it holds on any variety $X$ with symplectic singularities (see Definition \ref{def:symp-sing}). In the smooth case Boucksom actually proves it for pseudoeffective $\mathbb{R}$-classes, while here we restrict ourselves to effective $\mathbb{Q}$-divisors.   To state our result recall that, following Kirchner \cite{Kir}, one can define a quadratic form $q_{X,\sigma}$ on $H^2(X,\mathbb C)$, where $\sigma$ is a reflexive 2-form on $X$ which is symplectic on $X_{reg}$. When $X$ is a primitive symplectic variety (in the sense of Definition \ref{def:ISV}) and $\sigma$ is a 
normalized symplectic form (see Definition \ref{def:normed-q}) one shows that $q_{X,\sigma}$ is independent of $\sigma$ and is called the Beauville-Bogomolov-Fujiki quadratic form { (and denoted $q_X$)}. 
To deal with the non-$\Q$-factorial case we extend  the definition of $q_{X ,\sigma}$ to Weil divisors (cf. Definition \ref{def:q_on_Weil}). We refer the reader  to Section \ref{ss:q_on_Weil} for further details.

\begin{thm}\label{thm:q-Zar-intro}
If $X$ is a compact K\"ahler variety with symplectic singularities and $\sigma\in H^0(X,\Omega^{[2]}_X)$ a reflexive 2-form which is symplectic on $X_{reg}$, 
then all effective Weil $\Q$-divisors on $X$ have a unique $q_{X,\sigma}$-Zariski decomposition, i.e. $D$ can be written in a unique way as 
$$
 D= P(D) + N(D),
$$
where $P(D)$ and $N(D)$ are effective Weil $\Q$-divisors satisfying the following:
\begin{enumerate}
\item[1)] $P(D)$ is $q_{X,\sigma}$-nef, i.e. $q_{X,\sigma}(P(D),E)\geq 0$ for all effective Cartier divisors;
\item[2)] $N(D)$ is $q_{X,\sigma}$-exceptional, i.e.  the Gram matrix   $(q_{X}(N_i,N_j))_{i,j}$ of the irreducible components of the support of $N(D)$ 
is negative definite;
\item[3)] $q_{X,\sigma}(P(D),N(D))=0$;
\item[4)] For all integers $k\geq 0$ such that $kP(D)$ and $kD$ are integral, then the natural map 
$$
 H^0(X,\mathcal O_X(kP(D)))\twoheadrightarrow H^0(X,\mathcal O_X(kD))
$$
is a surjection.
\end{enumerate}

In particular, if $X$ is a primitive symplectic variety, any  $\Q$-Cartier effective divisor has a unique $q_X$-Zariski decomposition with respect to the Beauville-Bogomolov-Fujiki quadratic form $q_X$.
\end{thm}
The sheaves $\mathcal O_X(kP(D))$ and $\mathcal O_X(kD)$ above are the rank one reflexive sheaves associated to the Weil divisors $kP(D)$ and $kD$ (see Section \ref{ss:refl} for further details on Weil divisors and reflexive sheaves).   

Our proof goes along  the lines of Bauer's approach \cite{Bauer} to the Zariski decomposition for surfaces. As observed in \cite{Baueretal}, in order to have a Zariski-type decomposition on a compact K\"ahler variety it is sufficient to have a quadratic form on the second cohomology group (or on the divisor class group) that behaves like an intersection product, i.e., the intersection of two distinct prime divisors is always non-negative.  It is hence enough to prove that on any compact K\"ahler variety $X$ with symplectic singularities the quadratic form $q_{X,\sigma}$ introduced above behaves like an intersection product, see Theorem \ref{thm:gen}. 
Even if one wants to  deal only with $\Q$-Cartier divisors, without further hypotheses these may a priori have non-$\Q$-Cartier irreducible components. Our approach to the Boucksom-Zariski decomposition requires to evaluate the quadratic form on those components. This is a technical difficulty that we can circumvent, but it renders the proof more involuted. 

In the case of algebraic surfaces, as well as in the case of irreducible symplectic manifolds, the geometric significance of Zariski decompositions lies in the fact
that, given a pseudoeffective integral divisor $D$ with Zariski decomposition $D =P + N$,
one has for every sufficiently divisible integer $m > 1$ the equality
$$H^{0}(X, \mathcal{O}_{X}(mD)) = H^{0}(X, \mathcal{O}_{X}(mP)),$$ i.e., all sections in $H^{0}(X, \mathcal{O}_{X}(mD))$ come from global sections of the ``positive''   part  $\mathcal{O}_{X}(mP)$. The term
``sufficiently divisible'' here means that one needs to pass to a multiple $mD$ that clears denominators in $P$ for the statement to hold. In general, we do not know how to find numbers $m$ for a given surface $X$ and any integral pseudoeffective divisor. In \cite{BPS17}, the main result tells us that the question about the possible values of $m$ is strictly related to the bounded negativity conjecture, which is another open problem in the theory of linear series conjecturing the boundedness of negative self-intersections of irreducible divisors on any surface. Let us illustrate this phenomenon using  the   well-known case of smooth projective $K3$ surfaces. By the adjunction formula,  for any irreducible and reduced curve $C$ one has that $C^{2} \geq -2$, and in order to clear denominators in the Zariski decomposition for any pseudoeffective divisors $D$ one can take $m = 2^{\rho - 1}!$, where $\rho$ denotes the Picard number -- see \cite[Example 3.2]{BPS17} for details. 

It is natural to ask whether we can generalize the above considerations to the case of higher dimensional projective varieties. First of all, we have several variations on the classical Zariski decomposition, for instance the Cutkosky-Kawamata-Moriwaki-Zariski decomposition, but this decomposition, as it was shown by Cutkosky \cite{Cutkosky}, cannot exist in general. On the other hand, there is no natural and meaningful generalization of the bounded negativity conjecture in general. In \cite{lesottem}, the authors constructed an example of a sequence of irreducible and reduced effective divisors $D_{k}$ on smooth projective $3$-fold $Y$ such that $D_{k}^{3} \rightarrow -\infty$. 

However, as recalled before, we do have the Boucksom-Zariski decomposition on irreducible symplectic manifolds. 

In the case of {\it projective} irreducible symplectic manifolds, similarly to the case of $K3$ surfaces, thanks to the presence and properties of the Beauville-Bogomolov quadratic form we are able to prove an effective version of the analogue of the bounded negativity conjecture. More precisely in Proposition \ref{prop:BNChk} we prove that for a   projective     irreducible symplectic manifold $X$ the Beauville-Bogomolov self-intersection of any irreducible divisor is bounded from below by $4 \Card(A_X) $, where $\Card(A_X)$ is the  cardinality of the (finite) discriminant group 
$$A_X:=H^2(X,\mathbb Z)^\vee/H^2(X,\mathbb Z)$$
of the intersection lattice.
 The proof relies on two known results of Druel and Markman that we recall in Propositions \ref{prop:druel_prime}, \ref{prop:mark_prime}. 
 Notice that in the meantime these results have been proved in the singular framework, cf. \cite[Theorem 1.2, item (1)]{LMP} and \cite[Remark 3.11]{LMP2},  giving a lower bound on the Beauville-Bogomolov self-intersections of irreducible divisors in the singular case.
 For the boundedness of denominators in the Boucksom-Zariski decomposition we follow the lines of \cite{BPS17} to conclude with an effective bound  on denominators.
\begin{thm}\label{thm:bound-intro}
Let $X$ be a smooth projective irreducible symplectic variety of Picard number $\rho(X)$. The denominators of the coefficients of the negative and positive parts of the Boucksom-Zariski
decompositions of all pseudoeffective Cartier divisors   are  bounded by $(4\Card(A_X))^{\rho(X)-1} !$. 
\end{thm}
For the proof see Corollary \ref{corofZariskiBNC}.  Let us observe that $\rho(X)\leq h^{1,1}(X)$ and hence $(4\Card(A_X))^{h^{1,1}(X)-1} !$ gives a bound that is uniform for the whole family of deformations of $X$.

From Theorem \ref{thm:bound-intro} we easily deduce a result of effective birationality which is interesting on its own.
\begin{cor}\label{cor:eff-bir-bound}
Let $X$ be a smooth projective irreducible symplectic variety of dimension $2n$ and $L\in \Pic(X)$ a big line bundle on it. 
Then for all 
\begin{equation}\label{eq:eff} 
m\geq \frac{1}{2}(2n+2)(2n+3)(4\Card(A_X))^{\rho(X)-1} !
\end{equation}
the map associated to the linear system $|mL|$
is birational onto its image. 
\end{cor}
Again, by replacing $\rho(X)$ with $h^{1,1}(X)$ in equation (\ref{eq:eff}), we obtain an effective bound that holds for the whole family of deformations of $X$.
In particular, the corollary above answers affirmatively and effectively to a strong version of a question asked by Charles in the Introduction of \cite{Charles}. 
It is important to notice that even without answering this question Charles is able to obtain a birational boundedness result for smooth projective irreducible symplectic varieties, see \cite[Theorem 1.2]{Charles}. 
Such a result follows form Corollary \ref{cor:eff-bir-bound} by standard arguments, see Corollary \ref{cor:bir-bound} in Section 5. 
After the completion of the paper we were informed by C. Birkar that, for any $d>0$, he proved the existence of a bounded number $m$  giving the birationality of $|mL|$, where $L$ is a big integral divisor on any Calabi-Yau variety of fixed dimension $d$ with klt singularities (see \cite[Corollary 1.4]{Bi}). On the one hand, his result is obviously much more general than ours. On the other hand, the integer $m$ in his theorem is not explicit, while, in the case of smooth irreducible symplectic varieties, Corollary \ref{cor:eff-bir-bound} provides such an explicit bound.

{\it Added in the revision process:} Following exactly the same approach adopted here, using results from \cite{LMP2}, Corollary \ref{cor:eff-bir-bound} above has been extended to primitive symplectic varieties in \cite[Theorem 1.1]{LMP3}. 
\section{Preliminaries.}

We will use the word ``manifold'' to stress the smoothness, while we use the words ``variety'' when smoothness
is not required. 

\subsection{Irreducible symplectic varieties: the smooth case}
An irreducible symplectic manifold is a compact K\"ahler manifold $X$ which is simply connected and has a holomorphic symplectic
form $\sigma$ such that $H^0(X,\Omega^2_X)=\C\cdot \sigma$. For a general introduction to this subject, see \cite{Huy-basic}.

Let $X$ be an irreducible symplectic manifold of dimension $2n \geq 2$. Let $\sigma\in H^0(X,\Omega^2_X)$ such that 
$\int_{ X} \sigma^n \bar{\sigma}^{n}=1$.
Then, following \cite{Beau84},  the second cohomology group $H^2(X,\mathbb C)$ is endowed with a  quadratic form $q_{X}$ defined as follows
 $$
 q_{X}(a):=\frac{n}{2}\int_{ X} ( \sigma \bar{ \sigma})^{n-1}a^2 +
 (1-n)\bigg(\int_{ X} \sigma^n \bar{\sigma}^{n-1}a \bigg)\cdot
 \bigg(\int_{ X} \sigma^{n-1} \bar{ \sigma}^{n}a\bigg),\ a\in H^2(X,\mathbb C),
 $$
 which is non-degenerate and, up to a positive multiple, is induced by an integral non-divisible quadratic form on $H^2(X,\mathbb Z)$
 of signature $(3, b_{2}(X)-3)$. The form $q_X$ is called the Beauville-Bogomolov-Fujiki quadratic form of $X$. 
By Fujiki \cite{F},  there exists a positive rational number $c_X$ (i.e. the Fujiki constant of $X$) such that
$$c_X \cdot q_X^{n}(\alpha) = \int_{X}\alpha^{2m},\ \forall \alpha \in H^2(X,\mathbb Z).$$

\subsection{Irreducible symplectic varieties: the singular case}
Let $X$ be a normal complex variety.
Recall that, for any $p\geq 1$, the sheaf $\Omega_X^{[p]}$ of reflexive holomorphic $p$-forms on $X$
is $\iota_* \Omega_{X_{reg}}^{p}$, where 
$$\iota : X_{reg}\hookrightarrow X$$
is the inclusion of the regular locus of $X$. 
It can be alternatively (and equivalently) defined by the double dual $\Omega_X^{[p]}=(\Omega_X^{p})^{**}$.
Recall the following definition which is due to Beauville \cite{Beau00}. 
\begin{definition}\label{def:symp-sing}
Let X be a normal K\"ahler  variety.
\begin{enumerate}
\item[i)] A symplectic form on X is a closed reflexive 2-form $\sigma$ on $X$
which is non-degenerate at each point of $X_{reg}$.
\item[ii)]  If $\sigma$ is a symplectic form on $X$, the variety $X$ has symplectic
singularities if for one (hence for every) resolution $f : \tilde X \to X$ of the singularities
of $X$, the holomorphic symplectic form $\sigma_{reg} :=\sigma_{|X_{reg}}$ extends to
a holomorphic 2-form on $\tilde X$.
\end{enumerate}
\end{definition}

\begin{definition}\label{def:ISV}
A primitive symplectic variety is a normal compact K\"ahler variety $X$ such that $h^1(X,\mathcal O_X)=0$ and $H^0(X,\Omega_X^{[2]})$ is generated by a holomorphic symplectic form $\sigma$ such that $X$ has symplectic singularities.\end{definition}
For a normal variety $X$ such that $X_{reg}$ has a symplectic form $\sigma$, having canonical singularities is in fact equivalent to Beauville's condition above that the pullback of $\sigma$ to a resolution of $X$ extends as a regular $2$-form. By \cite{Elk81} and \cite[Corollary 1.7]{KS}, this is also equivalent to $X$ having rational singularities.
For the definition and basic properties of K\"ahler forms on possibly singular complex spaces we refer the reader to \cite[Section 2]{BL2}.

One can show that a smooth primitive symplectic variety $X$ is simply connected (see e.g. \cite[Theorem 1]{Sch1}), hence such $X$ is an irreducible symplectic manifold. 
In order to put Definition \ref{def:ISV} into perspective, recall first the following.
\begin{definition}\label{def:prim-ISV}
An irreducible symplectic variety is a normal  compact K\"ahler variety $X$ with canonical singularities
and with a symplectic form $\sigma$ such that for all quasi-\'etale morphisms $f:X'\to X$ the reflexive pull-back $f^{[*]}\sigma$ of $\sigma$ 
generates the exterior algebra of reflexive forms on $X'$. 
\end{definition}
Irreducible symplectic varieties appear in the Beauville-Bogomolov decomposition for minimal models with trivial canonical
class, obtained thanks to  contributions of several groups of people \cite{BGL, Cam20, Dru18, GGK, GKP16, HP}.  In the smooth case being primitive symplectic or irreducible symplectic is equivalent by \cite[Proposition 3]{Beau84}, while in the singular case an irreducible  symplectic variety is primitive symplectic, see e.g. \cite[Proposition 6.9]{GKP16}, but it is a more restrictive notion. Indeed, the Kummer singular surface $A_{/\pm1}$ is primitive symplectic but has a quasi-\'etale cover by $A$, hence it is not irreducible symplectic.  Moduli spaces of sheaves on a projective $K3$ surface (or the Albanese fiber of moduli spaces of sheaves on an abelian surface) are shown to be irreducible  symplectic varieties by Perego-Rapagnetta in \cite{PR}.

We now recall the work of Namikawa, Kirchner, Schwald and Bakker-Lehn on the Beauville-Bogomolov-Fujiki form in the singular setting. 
We start by recalling the following definition which make sense in general, but will be mainly relevant for symplectic varieties. 
We will only use it once outside this framework, in the proofs of Proposition \ref{prop:ind} and Theorem \ref{thm:gen}, where we will consider it on resolutions of varieties with 
symplectic singularities.

\begin{definition}[{\cite[Definition 3.2.1]{Kir}}]\label{def:q-kir}
Let $X$ be a $2n$ dimensional complex compact variety and $w\in H^2(X,\mathbb C)$ a non-zero class. 
 Then for all $a\in H^2(X,\mathbb C)$ we define a  quadratic form $q_{X,w}$ on $X$ as follows
 $$
 q_{X,w}(a):=\frac{n}{2}\int_{ X} ( w \bar{ w})^{n-1}a^2 +
 (1-n)\bigg(\int_{ X} w^n \bar{ w}^{n-1}a \bigg)\cdot
 \bigg(\int_{ X} w^{n-1} \bar{ w}^{n}a\bigg).
 $$
\end{definition}
In what follows whenever $X$ has symplectic singularities, in Kirchner's definition we will take $w$ to be equal to the symplectic form $\sigma$ on $X$.
When $X$ is a primitive symplectic variety, Namikawa defines also a quadratic form on $X$ by pulling everything back to the resolution of singularities of $X$. 
Kirchner's and Namikawa's  approaches are equivalent, as it is proved in \cite[Proposition 3.2.15]{Kir} -- see also \cite{Sch}. As noticed in \cite[Section 5.1]{BL2}, the hypothesis of the projectivity of $X$ that appears in \cite{Sch} is unnecessary. 

To obtain a uniquely determined quadratic form $q_X$ on a primitive
symplectic variety X, it is convenient to normalize the symplectic form $\sigma$. Let $I(\sigma):=\int_{ X} ( \sigma \bar{ \sigma})^n$ and consider
$s:=I(\sigma)^{-1/2n}\cdot \sigma$. Then $I(s)=1$ and
one verifies that,  once the normalization is performed, the resulting quadratic form 
$q_{X,s}$ does not depend on the choice of the symplectic form -- see \cite[Lemma 22]{Sch} and \cite[Lemma 5.3]{BL2}. 

\begin{definition}
\label{def:normed-q}
Let $X$ be a primitive symplectic variety. The Beauville-Bogomolov-Fujiki form of $X$
is defined as $q_X:=q_{X,\sigma}$ for any $\sigma\in H^0(X,\Omega_X^{[2]})$ with $I(\sigma)=1$. 

\end{definition}
As noticed in \cite[Lemma 5.5]{BL2}, there exists a real multiple of $q_X$ which takes integral values on integral classes. We finally collect in a single statement several observations and results due to Schwald and Bakker-Lehn which we are going to use. 

\begin{thm}\label{thm:useful}
\begin{enumerate} 
\item Let $f:Y\to X$ be a proper birational morphism between complex compact
varieties. Then $q_{X,w}(v)=q_{Y,f^*w}(f^*v)$ for all $v,w \in H^2(X,\mathbb C)$ -- see {\cite[Proof of Lemma 23]{Sch}} for details.

\item Let $X$ be a primitive symplectic variety of dimension $2n$. There exists a positive number $c_X\in \mathbb R_{>0}$ such that  for all $a\in H^2(X,\mathbb C)$ the following holds
$$ 
c_X\cdot q_X(a)^n=\int_Xa^{2n},
$$
please consult \cite[Theorem 2, item (1)]{Sch} and \cite[Proposition 5.15]{BL2} for details.
\item Let $X$ be a primitive symplectic variety. The restriction of $q_X$ to $H^2(X,\mathbb R)$ is a non-degenerate real quadratic form of signature $(3,b_2(X)-3)$ -- see \cite[Theorem 2 (2)]{Sch} and \cite[Lemma 5.3]{BL2} for details.
\end{enumerate}
\end{thm}

In the singular irreducible symplectic setting, the second cohomology carries a pure Hodge structure (see e.g. \cite[Corollary 3.5]{BL2}), therefore we can  define the positive cone  $\mathcal{P}$ as the connected component of $\{\alpha \in H^{1,1}(X,\mathbb{R}) : q(\alpha) > 0 \}$ containing a K\"ahler class,  and the movable cone $\mathcal{M}$ is defined as the closure of the cone generated by classes of
effective Cartier divisors $L$ such that the base locus of the linear series $|L|$ has codimension at least $2$. Moreover, we denote by $\mathcal{E}$ the pseudoeffective cone (which is closed and convex). In the case of irreducible symplectic manifolds, one can show that the dual of the pseudoeffective cone $\mathcal{E}^{*}$ coincides with the movable cone $\mathcal{M}$, see e.g. \cite[Theorem 7 and Remark 9]{HT}, \cite[Proposition 5.6]{mark_tor} or \cite[Lemma 2.7]{Den}. 

\subsection{Weil divisors and reflexive sheaves.}\label{ss:refl}
In this subsection we collect some known facts on Weil divisors and rank one reflexive sheaves. 
We refer the reader to \cite{Schwede} for the proofs, further details and references. 

If $\mathcal F$ is a coherent sheaf on a variety $X$, then its dual is $\mathcal F^\vee:={\mathcal Hom}_{\mathcal O_X}(\mathcal F, \mathcal O_X)$. There is a natural map from $\mathcal F$ to the double-dual $(\mathcal F^\vee)^\vee$ and the coherent sheaf is called {\it reflexive} if this map is an isomorphism. 

Given a (Weil) divisor $D$ on a normal variety $X$, we consider the coherent sheaf $\mathcal O_X (D)$ defined as follows:
$$\Gamma (V,\mathcal O_X(D))=\{f \in K(X): \div(f)_{|V} +D_{|V} \geq 0\}$$
for any open subset $V\subset X$. If  $D$ is a prime divisor, then $\mathcal O_X (-D) = \mathcal I_D$. Furthermore,
if $D$ is any divisor, then $\mathcal O_X (D)$ is a reflexive sheaf, cf. \cite[Proposition 3.4]{Schwede}. 

\begin{prop}\label{prop:1}
Let $X$ be a normal variety. 
\begin{enumerate}
\item Any reflexive rank 1 sheaf $\mathcal F$  is of the form $\mathcal O_X(D)$ for some Weil divisor $D$.
\item Two Weil divisors $D_1$ and $D_2$ are linearly equivalent if and only if $\mathcal O_X (D_1) = \mathcal O_X (D_2)$.
\item To every non-zero global section $s \in H^0 (X,\mathcal F)$ of a reflexive rank 1 sheaf $\mathcal F$, we can associate an effective divisor D on X.
\item Let $\mathcal F$ be a reflexive rank 1 sheaf. For every effective Weil divisor $D$ such that $\mathcal O_X (D) = \mathcal F$, there is a section $s \in H^0(X, \mathcal F )$ such that $s$ corresponds to $D$.
\item Two non-zero global sections $s_1,s_2\in H^0(X, \mathcal F )$ of a reflexive rank 1 sheaf $\mathcal F$ determine the same divisor if and only if there is a unit $u\in H^0(X,\mathcal O_X)$ such that $s_1 = us_2$.
\end{enumerate}
\end{prop}
For the proofs, see \cite[Propositions 3.7, 3.11 and 3.12]{Schwede}. 
We also have the following.
\begin{prop}[{\cite[Theorem 2.8 and Proposition 3.13]{Schwede}}] \label{prop:2}
Let $X$ be a normal variety. Let $D_1$ and $D_2$ be two Weil divisors on $X$. Then we have the following facts:
\begin{enumerate}
\item $\mathcal O_X (D_1+D_2)= (\mathcal O_X (D_1)\otimes \mathcal O_X (D_2))^{\vee\vee}$.
\item $\mathcal O_X (-D_1)=\mathcal O_X (D_1)^\vee$.
\end{enumerate}
\end{prop}
Thus one can turn the set of (isomorphism classes of) rank 1 reflexive sheaves into a group as follows. To add two sheaves, simply tensor them together and then double-dualize. To invert a rank 1 reflexive sheaf, simply dualize. The sheaf $\mathcal O_X$ is the identity. This group is clearly isomorphic to the divisor class group by the previous results.

Thanks to the previous propositions, it makes sense to talk about the complete linear system $|D|$ associated to a Weil divisor $D$ intended as the set of all Weil divisors $D'$ linearly equivalent to $D$ or, equivalently, to $\mathbb P (H^0(X, \mathcal O_X(D)))$. The base ideal $\mathfrak {b} (|D|)$ associated to $|D|$ is the image of the map 
$$
 H^0(X, \mathcal O_X(D))\otimes \mathcal O_X(D)^\vee\cong H^0(X, \mathcal O_X(D))\otimes \mathcal O_X(-D) \to \mathcal O_X
$$
induced by the evaluation morphism.
The base locus ${\bf Bs}(|D|)$ is the subscheme cut out by the base ideal $\mathfrak {b} (|D|)$.

\subsection{The quadratic form on Weil divisors.}\label{ss:q_on_Weil}
Even if one wants to deal only with $\Q$-Cartier divisors without further hypotheses these may a priori have non-$\Q$-Cartier irreducible components. The approach to the Boucksom-Zariski decomposition that we will present in the next section will require to evaluate the quadratic form defined in Definition \ref{def:q-kir} on those components. A  way out via a minimal $\Q$-factorialization is possible in the projective case thanks to the MMP (see Remark \ref{rmk:proj} for the details), but $\Q$-factorializations are not available yet in the non-projective case. 
Therefore, in order to deal with the most general framework (non-projective, non-$\Q$-factorial varieties), we extend the definition of the quadratic form to Weil divisors.

Let $X$ be a normal compact K\"ahler variety.  Let $\Sigma_0\subset X$ be the closed sublocus of points where the singularities of $X$ are worse than of ${\rm ADE}$-type, also known as the {\it dissident locus} of $X$  -- see \cite{mark_g2, Nami} for details. Through the rest of the paper, we will consider the complement open subset 
$$
 \iota : U:= X\setminus \Sigma_0 \hookrightarrow X. 
$$
Notice that $U$ is $\Q$-factorial since ${\rm ADE}$-type singularities are $\Q$-factorial. Moreover,
\begin{equation}\label{eq:codim_diss}
\codim_X(\Sigma_0)\geq 4
\end{equation}
by \cite[Proposition 1.6]{Nami}.
\begin{remark}\label{rmk:diss-not-diss}
Let $X$ be a normal compact K\"ahler variety of dimension $2n$ with symplectic singularities and let $\sigma$ be a symplectic form on it. Let $\pi : Y\to X$ be any resolution of singularities of $X$ and let $\sigma_{\pi}$ be the extension of $\sigma$ to $Y$. 

We claim that the form $\sigma_{\pi}^ {n-1}$ is trivial when restricted (i.e., pulled back via the inclusion map) to the smooth locus of any component of the preimage of the dissident locus $\Sigma_0$. Indeed, let $Z$ be any  component of the preimage of the dissident locus inside $Y$. Note that it is enough to prove that $\sigma_{\pi}^{n-1}$ is trivial after restriction to a dense Zariski open subset of the smooth locus of $Z$. Let $V=\pi(Z)$ and $\pi_V\,:Z\to V$ be the restriction of $\pi$. Since $\pi$ and hence also $\pi_V$ is proper, after possibly taking a resolution of $Z$ and base changing to a dense Zariski  open subset of $V$, we can assume that both $\pi_V: Z\to V$ and $V$ are smooth. Therefore, by \cite[Lemma 2.9]{kal} (cf. \cite[Remark 3.6]{BL2} for the analytic setting) applied to the following diagram (which is analogous to \cite[Diagram (2.2)]{kal}):
$$    
\xymatrix{
Z\ar[r]\ar[d]^{\pi_V} & Y\ar[d]^{\pi}\\
V\ar[r] & X}
$$
we have that $\sigma_{\pi}|_{Z}$ is the pullback of a two form $\omega_V$ over $V$. As $V$ has codimension at least four, we have that $$(\sigma_{\pi}^{n-1})|_Z=(\sigma_{\pi}|_Z)^{n-1}=(\pi_V^*(\omega_V))^{n-1}=\pi_V^*(\omega_V^{n-1})=0.$$
In this way we have proven that $\sigma_{\pi}^{n-1}$ is trivial on a Zariski open subset of a chosen component of the preimage of the dissident locus. We conclude that $\sigma_{\pi}^{n-1}$ is trivial on the smooth locus of that component.

Furthermore, if $\pi$ is a crepant resolution of singularities, $\sigma_\pi$ is a symplectic form when restricted to the preimage of the complement $U$ of $\Sigma_0$. Finally, if $E$ is any divisor over a smooth point of $X$, then the form $\sigma_{\pi}$ is zero on $E$. 
\end{remark}
We will use the notation $\operatorname{Div}(X)$ for the group of Weil divisors on $X$ without any equivalence relation while 
$\operatorname{Div}_{\mathbb Q}(X):=\operatorname{Div}(X)\otimes \mathbb Q$ will be the group of Weil $\mathbb Q$-divisors.
Let $D\in \operatorname{Div}(X)$ be a Weil divisor on $X$ and $D_U$ its restriction to $U$. 

Note that the open subset $U$ is $\Q$-factorial as it has only ${\rm ADE}$-type singularities, hence the Weil divisor $D_{U}$ is $\Q$-Cartier.

Let $\pi :Y\to X$ be any resolution of singularities and $\pi_U$ its restriction to the preimage of $U$. Since $D_U$ is a $\Q$-Cartier divisor there exists $k$ such that $kD_U$ is a Cartier divisor. Now since the image of $\pi$ is not contained in the support of $kD_U$, we can define a Cartier divisor $\pi_U^* (kD_U)$ (not just a divisor class) by pulling back the local equations defining $kD_U$. We set $\pi_U^*D_U:=\frac{1}{k} \pi_U^* (kD_U)$  which is a $\Q$-Cartier $\Q$-divisor and hence it is also a Weil $\mathbb Q$-divisor on $\pi^{-1}(U)$. The latter, abusing the notation, is also denoted by $\pi_U^* D_U$. Note that for Weil $\mathbb Q$-divisors on $\pi^{-1}(U)$ one can define a closure map  $${\iota_{\pi^{-1}(U)}}_*: \operatorname{Div}_{\mathbb Q}(\pi^{-1}(U)) \to \operatorname{Div}_{\mathbb Q}(Y)$$ associated to the inclusion $\iota_{\pi^{-1}(U)}\colon \pi^{-1}(U)\to Y$ by extending to divisors the map associating to a subvariety of codimension $1$ of $\pi^{-1}(U)$ its closure in $Y$. However, this map is not compatible with linear equivalence. Nonetheless, we can define $\pi^{\dagger}D$ to be the Poincar\'e dual of the homology class of the closure of the Weil $\mathbb Q$-divisor $\pi_U^* D_U$ by the inclusion $\iota_{\pi^{-1}(U)}\colon \pi^{-1}(U)\to Y$. 
More precisely, we have
\begin{equation}\label{eq:dagger}
  \pi^{\dagger}D:=[(\iota_{\pi^{-1}(U)})_* (\pi_{U}^*D_U)]^{\vee}\in H^2(Y,\Z).  
\end{equation}

Note that for an effective Cartier divisor what we are doing amounts to taking the total transform of the divisor restricted to $U$, and then closing it up, so that we obtain a strict transform over the dissident locus. Note also that since the closure map we used is not compatible with linear equivalence on Weil $\mathbb Q $-divisors,   $\pi^{\dagger}$ is not compatible with linear equivalence of Weil divisors. Nonetheless, we will see in Proposition \ref{prop:ind} that the following definition produces a quadratic form that is well-defined on linear equivalence classes of Weil divisors.

\begin{definition}\label{def:q_on_Weil}
Let $X$ be a $2n$ dimensional compact K\"ahler variety with symplectic singularities 
and let $\sigma$ be a symplectic form. 
Let $U\subset X$ be the open subset defined as above. For any Weil divisor $D$ on $X$ we define 
$q_{X,\sigma}^{\textrm{Weil}}(D)$ as follows
$$
 q_{X,\sigma}^{\textrm{Weil}}(D):= q_{Y,[\sigma_{\pi}]}(\pi^{\dagger}D),
$$
where $\pi :Y\to X$ is any resolution of singularities, $[ \sigma_{\pi}]\in H^2(Y, \mathbb C)$ is the cohomology class of the extension of $\sigma$ to $Y$, 
and $q_{Y,[\sigma_{\pi}]}$ is as in Definition \ref{def:q-kir}.
\end{definition}

\begin{prop}\label{prop:ind}
Let $X$ be a $2n$ dimensional compact K\"ahler variety with symplectic singularities and let $\sigma$ be a symplectic form. 
\begin{enumerate}
    \item[a)]  Definition \ref{def:q_on_Weil} above does not depend on the choice of the resolution of singularities of $X$.
    \item[b)] If $X$ is a compact K\"ahler variety with symplectic singularities and $D$ is a Cartier divisor, then Definitions \ref{def:q_on_Weil} and \ref{def:q-kir}
do coincide.
 \item[c)] Definition \ref{def:q_on_Weil} above does not depend on the choice of the divisor $D$ in the same linear equivalence class.  
\end{enumerate}
\end{prop}
\begin{proof}
a) Let $\pi_j :Y_j\to X,\ j=1,2,$ be two resolutions of singularities. Consider a smooth birational model $\tilde Y$
dominating both $Y_1$ and $Y_2$ and sitting in the following commutative diagram:
\begin{equation}\label{eq:comm}
\xymatrix{
& \tilde Y \ar[dr]^{\tilde\pi_2}\ar[dl]_{\tilde\pi_1}\ar[dd]^{\nu} & \\
Y_1\ar[dr]_{\pi_1} & &Y_2\ar[dl]^{\pi_2}\\
& X &}
\end{equation}
Let $\sigma$ be a symplectic form.
Then for any Weil divisor $D$ on $X$ and $j=1,2$ we have
$$
q_{Y_j,[\sigma_{\pi_j}]}(\pi_j^{\dagger}D )=
q_{\tilde Y,\tilde\pi_j^*[\sigma_{\pi_j}]}(\tilde\pi_j^*(\pi_j^{\dagger}D))=
q_{\tilde Y,[\sigma_{\nu}]}(\nu^{\dagger}(D)),
$$
where the first equality follows from  Theorem \ref{thm:useful}, item (1), and the last from the fact that by the commutativity of the diagram (\ref{eq:comm}) the two classes 
$[\sigma_{\nu}]$ and $\tilde\pi_j^*[\sigma_{\pi_j}]$ coincide and $\pi_j^{\dagger}(D)$ and $\nu^{\dagger}(D)$ differ only by divisors supported on the preimage by $\nu$ of the dissident locus $\Sigma_0$ and that difference does not affect the integrals in Definition \ref{def:q-kir}.
Indeed, for any exceptional divisor $E$ over the dissident locus we have
\begin{equation}\label{eq:q=0}
q_{Y, [\sigma_{\nu}]}(E)= q_{Y, [\sigma_{\nu}]}(\nu^{\dagger}D, E) = 0
\end{equation}
because both bilinear forms are computed by an integral over $E$ involving $\sigma_{\nu}^{n-1}$, which can be computed after removing the zero measure subset of singular points of $E$. Then, by Remark \ref{rmk:diss-not-diss}, the form $\sigma_{\nu}^{n-1}$ is trivial after restriction to the smooth locus of $E$.

b) Let $\pi : Y\to X$ be a resolution of singularities and  let $\sigma\in H^2(X,\mathbb C)$ be the class of a symplectic form.
We first note that 
\begin{equation}\label{eq:Ei}
\pi^* D= \pi^{\dagger}D + \sum a_iE_i,
\end{equation}
where the $E_i$'s are $\pi$-exceptional divisors whose image is supported on $\Sigma_0$.

Thus, from (\ref{eq:Ei}) and (\ref{eq:q=0}) we deduce
$$
 q_{X,\sigma}(D)=q_{Y,\pi^*\sigma}(\pi^*D)= q_{Y,\pi^*\sigma}(\pi^\dagger D) = q_{X,\sigma}^{\textrm{Weil}}(D),
$$
where the first  equality follows again 
from Schwald's result Theorem \ref{thm:useful}, item (1).

c) Let $D$ and $D'$ be two linearly equivalent Weil divisors on $X$ and let $\pi : Y\to X$ be a resolution of singularities. Notice that, by definition of $\pi^\dagger$, the support of $\pi^\dagger(D)-\pi^\dagger(D')$ is contained in the components of the exceptional locus mapping to $\Sigma_0$, and then from \eqref{eq:q=0} we obtain the last item.  
\end{proof}
From now on we will therefore drop the ``Weil'' exponent in the notation introduced in Definition \ref{def:q_on_Weil} and  write $q_{X,\sigma}$. 
If $\mathfrak{C}$ is an interval then  clearly two open sets covering $\mathfrak{C}$ have non-empty intersection.

\section{Boucksom-Zariski decompositions on primitive symplectic  varieties}

We prove that the generalization of the Beauville-Bogomolov-Fujiki form to primitive symplectic varieties  behaves like an intersection product in the sense below and therefore by \cite{Baueretal} it allows to obtain a Boucksom-Zariski decomposition.  A technical difficulty comes from fact that the varieties need not be $\Q$-factorial.

We start by recalling some definitions.
 \begin{definition}\label{def:intersection product} Let $q$ be a quadratic form on a vector space $V$ and let $\mathcal B$ a basis of $V$.  We say that $q$ is an intersection product with respect to $\mathcal B$ if for any $D,D'\in \mathcal B$ such that $D\neq D'$ we have
\begin{equation}\label{eq:star}
q(D,D')\geq 0. 
\end{equation}
 \end{definition}

\begin{definition} \label{prime} 
Let $X$ be a normal compact complex space endowed with a quadratic form $q_{X}$ on $\Div_{\Q}(X)=\Div(X)\otimes \Q$, where $\Div(X)$ denotes the group of Weil divisors (without any equivalence relation).

A reduced and irreducible effective divisor $D\subset X$ is called a prime divisor. An effective Weil $\Q$-divisor $E=\sum a_i E_i$ 
is called $q_{X}$-exceptional if the Gram matrix   $(q_{X}(E_i,E_j))_{i,j}$ of the irreducible components of the support of $E$ 
is negative definite.  
Furthermore, we say that a Weil $\Q$-divisor $D$ is $q_{X}$-nef if 
$$q_{X}(D,E)\geq 0 $$ 
for every effective 
Weil $\Q$-divisor $E$, or equivalently for every prime divisor $E$.
In this context, we will say that $q_X$ is an intersection product (without mentioning any basis) if it is an intersection product with respect to the basis consisting of all prime divisors.
\end{definition}
 \begin{remark}
    Notice that by \cite[Proposition 4.2, item (ii)]{Boucksom1} the Beauville-Bogomolov-Fujiki quadratic form is an intersection product compatible with the linear equivalence on $\Div_{\Q}(X)$ with $X$ being a smooth irreducible symplectic variety. 

\end{remark}

\begin{definition}\label{definition:q-Zar}
Let $X$ be a normal compact K\"ahler variety endowed with a quadratic form $q_{X}$ on $\Div_{\Q}(X)$. Let $D$ be an effective Weil $\Q$-divisor on $X$.
A  rational  $q_{X}$-Zariski decomposition for $D$ is a decomposition 
$$
 D= P(D) + N(D),
$$
where $P(D)$ and $N(D)$ are effective Weil  $\Q$-divisors satisfying the following:
\begin{enumerate}
\item[1)] $P(D)$ is $q_{X}$-nef;
\item[2)] $N(D)$ is $q_{X}$-exceptional or trivial;
\item[3)] $q_{X}(P(D),N(D))=0$;
\end{enumerate}
The divisors $P(D)$ and $N(D)$ are called the positive and negative parts of $D$ respectively.
\end{definition}
Notice that even if the divisor $D$ is Cartier we do not require the positive and the negative parts of $D$ to be 
$\Q$-Cartier.

Let us formulate the following theorem.
\begin{thm}\label{thm:q-Zar}
Let $X$ be a normal compact K\"ahler variety endowed with a quadratic form $q_{X}$ on  $\Div_{\Q}(X)$ which is an intersection product  (in the sense of Definition \ref{prime}).
Then all effective Weil $\Q$-divisors on $X$ have a unique rational  $q_{X}$-Zariski decomposition. Moreover, if $q_X$ is compatible with linear equivalence then it also induces a unique $q_X$-Zariski decomposition on $\Cl_{\Q}(X)$ and both decompositions are compatible.
\end{thm}
\begin{proof}The theorem follows from \cite[Theorem 3.3]{Baueretal} by taking $V$ to be the (infinite dimensional) $\mathbb{Q}$-vector space $\Div_{\mathbb Q}(X)$, which is generated by prime Weil divisors. 
\end{proof}
 \begin{remark} In fact, in Theorem \ref{thm:q-Zar} the Zariski decomposition can be performed for $q_X$ being an intersection product in different ambient spaces $V$ and bases $\mathcal B$ (cf. Definition \ref{def:intersection product}). Once $q_X$ is compatible with linear equivalence, we can consider $D$  to be either an effective $\mathbb{Q}$-divisor or a linear equivalence class of such divisors or their  numerical class. For each such $D$ one can choose a basis consisting of classes of prime divisors in which the class of $D$ is presented with positive coefficients and perform a decomposition with respect to the chosen basis.  By the uniqueness of Zariski decomposition in all these contexts all these decompositions are compatible. More precisely, if $[D]$ is a numerical class (or linear equivalence class) of an effective $\mathbb{Q}$-divisor $D$ then there exists a unique decomposition $[D]=[P]+[N]$ into numerical classes (or linear equivalence classes)  of effective (with respect to our chosen basis) $\mathbb{Q}$-divisors  and it satisfies $[P]=[P(D)]$, $[N]=[N(D)]$. 
\end{remark}

\begin{remark}
If $X$ is $\Q$-factorial, then of course $P(D)$ and $N(D)$ are $\Q$-Cartier. In the general case, if every component of the divisor $D$ is $\Q$-Cartier, then the proof above shows that $P(D)$ and $N(D)$ are $\Q$-Cartier. However, if we only assume that the divisor $D$ is $\Q$-Cartier, then even in the case of irreducible symplectic varieties we do not know whether $P(D)$ and $N(D)$ must automatically be $\Q$-Cartier or not.
\end{remark} 

 \begin{prop} Let $q_X$ be an intersection product on $\Div_{\Q}(X)$ compatible with linear equivalence. Then the unique $q_X$-Zariski decomposition on $\Cl_{\Q}(X)$ from Theorem \ref{thm:q-Zar} satisfies the following additional condition: for all integers $k\geq 0$ such that $kP(D)$ and $kD$ are integral, the natural map 
$$
 H^0(X,\mathcal O_X(kP(D)))\to H^0(X,\mathcal O_X (kD))
$$
is an isomorphism, where $\mathcal O_X(kP(D))$ and $\mathcal O_X(kD)$ are the 
rank one reflexive sheaves associated to $kP(D)$ and $kD$, see Section \ref{ss:refl}.
\end{prop}

\begin{proof}
Let $k$ be a positive integer such that $kD$ and $kP(D)$ are integral and consider $M \in |kD|$. 
Then $M\sim k (P(D) + N(D))$. Let $N(D)=\sum_i a_i N_i$, with $a_i>0$, then by item 3) of Definition \ref{definition:q-Zar} there exists an index $i$ such that 
$$
q_X (M, N_i)= kq_X (N(D), N_i)<0.
$$
Therefore, since $q_X$ is an intersection product, we have that
\begin{equation}\label{eq:negative_support}
 N_i \subset \supp(M).
\end{equation}
By letting $M$ move in the linear system we conclude that for this $i$
$$
 N_i \subset \mathbf {Bs}(|kD|),
$$
where $\mathbf{Bs}(\cdot)$ denotes the  base locus. 

We thus have 
$$
 H^0(X,\mathcal O_X(kD-N_i))=H^0(X,(\mathcal O_X(kD)\otimes \mathcal I_{N_i})^{\vee \vee})\twoheadrightarrow H^0(X,\mathcal O_X(kD)),
$$
where the first equality follows from Proposition \ref{prop:2}.
Consider now  the decomposition $M-N_i= kP(D)+ (kN(D)-N_i)$. It is clearly a $q_{X}$-Zariski decomposition. We can hence repeat the above argument for $M_1=kD-N_i$. Iterating this process one obtains a sequence of divisors $M_j$ the last of which being $kP(D)$ and a sequence of surjections whose composition is the desired surjection
$$
 H^0(X,\mathcal O_X(kP(D)))\twoheadrightarrow H^0(X,\mathcal O_X(kD)).
$$

\end{proof}

 Notice that our proof of Theorem \ref{thm:q-Zar}  above leads to a slightly stronger statement on Zariski decompositions.
\begin{cor} Let $q_X$ be an intersection product on $\Div_{\Q}(X)$ compatible with linear equivalence.  The decomposition  $D= P(D) + N(D)$ provided by Theorem \ref{thm:q-Zar} for an effective Weil $\Q$-divisor $D$  satisfies $\mathrm{Supp} (P(D))\cup \mathrm{Supp} (N(D))=\mathrm{Supp} (D)$. In particular, for the $q_X$-Zariski decomposition $[D]=[P]+[N] $ in $\Cl_{\Q}(X)$, we can choose $P(D)\sim_{\mathrm{lin}} P$ and  $N(D)\sim_{\mathrm{lin}} N$ such that $\mathrm{Supp} (P(D))\cup \mathrm{Supp} (N(D))=\mathrm{Supp} (D)$.
\end{cor}
\begin{proof}
The inclusion $\mathrm{Supp} (P(D))\subset \mathrm{Supp}(D)$ is given by the natural map of sections, while the inclusion $\mathrm{Supp} (N(D))\subset \mathrm{Supp}(D)$ is given by \eqref{eq:negative_support}. The other inclusion follows from the fact that both $P(D)$ and $N(D)$ are effective.
\end{proof}

Recall that if a bilinear form satisfies condition \eqref{eq:star} on distinct prime  divisors we say that it is an  intersection product.
\begin{thm}\label{thm:gen}
Let $X$ is a $2n$-dimensional compact K\"ahler variety with symplectic singularities and $\sigma\in H^0(\Omega^{[2]}_X)$ a reflexive 2-form which is symplectic on $X_{reg}$. 
Then  the quadratic form $q_{X,\sigma}$ defined in Definition \ref{def:q_on_Weil} is an intersection product on $\Div_{\Q}(X)$ compatible with linear equivalence.
\end{thm}
\begin{proof}

We start by noticing that the compatibility with linear equivalence is exactly the content of Proposition \ref{prop:ind}, item (c).

Let $ \pi: \tilde X\to X$ be a resolution of singularities and $\tilde \sigma:=\sigma_{\pi}$ be the holomorphic $2$-form on $\tilde X$ extending $\sigma$, which exists by \cite[Theorem 1.4]{KS}. 
 
We first show that condition (\ref{eq:star}) holds for $q_{\tilde X,\tilde \sigma}$ on $\tilde X$ and then deduce from this that it also holds on $X$. To prove that (\ref{eq:star}) holds for $q_{\tilde X,\tilde \sigma}$,  we will argue as in  \cite{Boucksom1}. First of all, we notice that if $D$ and $D'$ are two distinct effective prime divisors on $\tilde X$  then the (numerical equivalence) class $\{D \cdot D'\}$ contains the closed positive $(2,2)$-current given by the integration along the effective intersection cycle $D\cap D'$.  Moreover, the form $(\tilde \sigma \overline {\tilde \sigma})^{n-1}$ is a smooth positive form (of bidimension $(2,2)$). Then, from Definition \ref{def:q-kir}, we immediately deduce that
$$
q_{\tilde X,\tilde \sigma}(D,D')= \frac{n}{2}\int_{D\cap D'} (\tilde\sigma \overline{ \tilde\sigma})^{n-1}
$$
and the latter is $\geq 0$ by the above preparation. Since condition (\ref{eq:star}) holds for $q_{\tilde X,{\tilde \sigma}}$ on $\tilde X$ the $q_{\tilde X,{\tilde \sigma}}$-Zariski decomposition holds for effective divisors on $\tilde X$ by Theorem \ref{thm:q-Zar}. 

To deduce that condition (\ref{eq:star}) holds for $q_{X,\sigma}$ on $X$, we argue as follows. 
As is Section \ref{ss:q_on_Weil} we will consider the open subset
$U=X\setminus \Sigma_0$ where $\Sigma_0$ is the dissident locus. 
Let $D$ and $D'$ be two distinct prime Weil divisors on $X$ and $[D_U], [D'_U]$ their restrictions (as cohomology classes since $U$ is $\Q$-Cartier) to $U$. We will use the notation $\pi^\dagger D$ as in \eqref{eq:dagger}. 
 
Consider the $q_{\tilde X, \tilde \sigma}$-Zariski decomposition on $\tilde X$ 
$$
 \pi^{\dagger} D' =P(\pi^{\dagger} D') + N(\pi^{\dagger}D')
$$
with $P':= P(\pi^{\dagger}D')$ and $N':=N(\pi^{\dagger}D')=\sum c_i N'_i$ satisfying all the items of Definition \ref{definition:q-Zar}. 
Then 
$$
 q_{ X, \sigma}(D,D')=q_{\tilde X,\tilde \sigma} (\pi^{\dagger}D, \pi^{\dagger} D')=q_{\tilde X,\tilde \sigma} (\pi^{\dagger}D, P') + \sum c_i q_{\tilde X,\tilde \sigma} (\pi^{\dagger} D,N'_i), 
$$
where the first equality follows from Definition \ref{def:q_on_Weil}. 
Notice that 
$$
 q_{\tilde X,\tilde \sigma} (\pi^{\dagger}D, P')\geq 0
$$
because $\pi^{\dagger}D$ is effective as it is the Poincar\'e dual to an effective homology class and, by Theorem \ref{thm:q-Zar}, the positive part $P'$ is 
$q_{\tilde X,\tilde \sigma}$-nef. 

For a component $N'_i$ which is not a $\pi$-exceptional divisor, we have that $$q_{\tilde X,\tilde \sigma} (\pi^{\dagger} D,N'_i)\geq 0$$ because $D$ and $D'$ are distinct and we have already proved that  (\ref{eq:star}) holds on $\tilde X$. Now if a component $N'_i$ is a $\pi$-exceptional divisor we have two cases. Either $\pi(N'_i)$ is contained in $\Sigma_0$ and then $q_{\tilde X,\tilde \sigma} (\pi^{\dagger}D,N'_i)=0$ as observed in Proposition \ref{prop:ind}, equation (\ref{eq:q=0}).
If $\pi(N'_i)$ is {\it not} contained in $\Sigma_0$,
then by a careful application of the push-pull formula, we get
\begin{eqnarray*}
q_{\tilde X,\tilde \sigma}(\pi^{\dagger}D,N'_i)=&&  \int_{\tilde X} (\tilde\sigma \overline{ \tilde\sigma})^{n-1}[\pi^{\dagger}D] [N'_i] =
 \int_{\pi^{-1}(U)}\pi_U^*(\sigma \bar \sigma)^{n-1}_{|\pi^{-1}(U)} 
 \pi^*_{U}[D_{U}][ N'_i|_{\pi^{-1}(U)}] =
\\  =&&\int_{ U}(\sigma \bar \sigma)^{n-1}[ D_U] \pi_{U*}[N'_i|_{\pi^{-1}(U)}],
\end{eqnarray*}
where $\pi_U:=\pi_{|\pi^{-1}(U)}: \pi^{-1}(U)\to U$. Indeed, the second equality holds by restriction of the integral to the dense open subset $U$ and the equality
$[\pi^{\dagger}D]|_{\pi^{-1}(U)}=\pi_{U}^*[D_{U}]$.  The last equality follows from the push-pull formula applied to the morphism $\pi_U$ which has proper fibres. Then the last integral is zero
as $N'_i$ is contracted by $\pi$ and we get that $q_{\tilde X,\tilde \sigma} (\pi^*D,N'_i)=0$ also in this case.
Therefore
$$
 q_{ X, \sigma}(D,D')\geq 0 
$$
 and we are done.  
  \end{proof}

\begin{proof}[Proof of Theorem \ref{thm:q-Zar-intro}]
It follows immediately from the combination of Theorems \ref{thm:q-Zar} and \ref{thm:gen}.
\end{proof}
\begin{remark}\label{rem nef = closure bir kahler}
Note that if $X$ is an irreducible symplectic manifold and $q_X$ is the Beauville-Bogomolov-Fujiki form then we know that closed cone of $q_X$-nef classes coincides with the closure of the birational K\"ahler cone (see e.g. \cite[Proposition 5.6]{mark_tor}).  
\end{remark}

\begin{remark}
It is important to notice that our proof of the existence of the Boucksom-Zariski decomposition in the
singular case is very different from the one in the smooth case due to Boucksom. Indeed, as mentioned in the introduction,
Boucksom first shows the existence of a general divisorial Zariski decomposition on a compact complex manifold $X$ (and, contrary to Bauer's approach, this is done by defining directly the negative part by attaching asymptotically defined multiplicities to components of the stable base locus of $D$) and then characterizes it when $X$ is a smooth irreducible symplectic variety as a Zariski decomposition with respect to the Beauville-Bogomolov-Fujiki intersection form. If one wants to follow the same path in the singular setting the first thing to do is to extend to this framework the general divisorial Zariski decomposition. This is done in \cite[Section 4]{BH} in the projective case. In the non-projective case things seem to be more subtle, though experts  believe it should work the same (for the 3-dimensional case see e.g. \cite[Section 4.B]{HPmmp}). Still, even if we had a general divisorial Zariski decomposition in the singular case its characterization as a Zariski decomposition with respect to the Beauville-Bogomolov-Fujiki form uses several results (see \cite[Section 4]{Boucksom1}) for smooth irreducible symplectic varieties not available yet in the singular case. 
\end{remark}
\begin{remark}\label{rmk:proj}
If $X$ is a projective variety with symplectic singularities and $\sigma$ is a symplectic form on it, there is an alternative and easier way to deal with non-$\Q$-Cartier components of $\Q$-Cartier divisors. Indeed, by \cite[Lemma 10.2]{BCHM} there exists a small birational morphism $\phi:Y\to X$ from a $\Q$-factorial projective variety $Y$ onto $X$. Then for any Weil divisor $D$ on $X$ we can define $q_{X,\sigma}(D)$ as $q_{Y,\phi^*\sigma}(\phi^{-1}D)$. By the smallness of $\phi$, if $D$ is a $\Q$-Cartier divisor we have that $\phi^{-1}(D)=\phi^*D$ and therefore the previous definition is consistent with the old one thanks to Theorem \ref{thm:useful}, item (1). 
Unfortunately, the analogous existence result in the non-projective case is not known, so we were forced to use our more complicated approach presented in Section \ref{ss:q_on_Weil}.
In the projective case, however, the two approaches are equivalent by Proposition \ref{prop:ind}.
\end{remark}

\section{Bounding denominators and bounded negativity}
Now we would like to present an effective way to bound coefficients in Zariski decompositions for (pseudo)effective divisors on any projective irreducible symplectic variety $X$.
By Theorem \ref{thm:q-Zar-intro}, any effective Cartier divisor $D$ on $X$  can be uniquely presented as
$$D = P + \sum_{i=1}^{k} a_{i}N_{i}, \quad  a_{i} \in \mathbb{Q}_{>0}$$
with $P$ and the $N_i$'s as in Definition \ref{definition:q-Zar} (if $X$ is smooth this also holds for any pseudoeffective divisor).
Notice that  we have  $k \leq \rho(X)-1 \leq h^2(X)-3$. Indeed, by definition of the Zariski decomposition,  the Gram matrix $[q_X(N_i,N_j)]_{i,j}$ is negative definite whereas the signature of the Beauville-Bogomolov form on $\operatorname{NS}(X)$ is $(1,\rho(X)-1)$. Then necessarily $k\leq \mathrm{rk} \, {\rm NS}(X)-1=\rho(X)-1$. 

It is natural to ask whether there exists an integer
$d(X) > 1$ such that for every effective integral divisor $D$ the denominators in the
Zariski decomposition of $D$ are bounded from above by $d(X)$.
If such a bound $d(X)$ exists, then we say that X has {\it bounded Zariski denominators}.

From now on, to lighten the notation, we set $q=q_X$. 
Consider the following system of linear equations which is given by taking $q(D,N_{j})$ for each $j \in \{1, ...,k\}$, namely
$$q(D,N_{j}) = q(P,N_{j}) + \sum_{i=1}^{k} a_{i} q(N_{i},N_{j}) = \sum_{i=1}^{k} a_{i} q(N_{i},N_{j}).$$
Rewriting the above line in a compact way, we obtain
\begin{equation}\label{eq:matrix}
(q(D,N_{1}), ...,q(D,N_{k}))^{T} = \bigg[ q(N_{i},N_{j}) \bigg]_{i,j=1,\ldots, k} \cdot (a_{1}, ...,a_{k})^{T}.
\end{equation}
\begin{lemma}\label{lem:a_i}
In the notation as above, we have 
$$a_{i} = \frac{\det S_i}{\det \bigg[q(N_{i},N_{j}) \bigg]_{k \times k}}$$
where $S_i$ is the $(k\times k)$ matrix  obtained from $\bigg[q(N_{i},N_{j})\bigg]_{k \times k}$ by replacing its $i$-th column with the vector $(q(D,N_{1}), ...,q(D,N_{k}))^{T}$.
\end{lemma}
\begin{proof}
 The formula for the $a_{i}$'s follows from (\ref{eq:matrix}) using Cramer's rule. 
 Moreover, notice that the $a_{i}$'s are positive since $q(D,N_{i}) < 0$ for each $i \in \{1, ...,k\}$ -- otherwise $N_{i}$ would not be in the negative part of $D$, for instance by \cite[Lemma 4.9]{Boucksom1}. 
\end{proof} 
 
 By the above considerations, we obtained an upper bound on the denominators of the coefficients $a_{i}$ by the value $ | \det [q(N_{i},N_{j})]_{k \times k} |$. This gives a relation between boundedness of denominators in the Boucksom-Zariski decompositions and the boundedness of the squares of prime exceptional divisors.  A similar phenomenon was observed in the case of algebraic surfaces in \cite{BPS17}. This observation might not be very surprising, mostly due to the fact that one expects that irreducible symplectic  manifolds behave like surfaces for what concerns linear series once one replaces the intersection form with the Beauville-Bogomolov-Fujiki quadratic form.
The precise statement is the following:
\begin{thm}\label{ZariskiBNC}
Let $X$ be an irreducible symplectic variety of dimension $2n\geq 2$.
 Then the following two conditions are equivalent:
\begin{enumerate}
    \item[1)] the self-intersection numbers with respect to Bogomolov-Beauville form $q$ of integral prime divisors are bounded from below;
    \item[2)] the denominators of the coefficients $a_{i}$ of the negative parts in the Boucksom-Zariski decompositions of (pseudo)effective Cartier divisors are bounded by a global constant.

\end{enumerate}
\end{thm}

\begin{proof}[Proof of Theorem \ref{ZariskiBNC}]
For the proof one proceeds along the same lines as in \cite{BPS17} replacing the intersection form on the surface with the Beauville-Bogomolov-Fujiki quadratic form.
\end{proof}

\begin{remark}
Let us come back to the introduction, if $X$ is a $K3$ surface, then by the adjunction formula the self-intersection numbers of prime curves are bounded by $-2$. By \cite[Example 3.2]{BPS17}, the denominators of coefficients in Zariski decompositions are bounded from above by $2^{\rho - 1}!$, where $\rho$ is equal to the Picard number, so from now on we assume that $n>1$.
\end{remark}

\begin{remark}
If $X$ is smooth in item 2) of Theorem \ref{ZariskiBNC} then we can take the divisors to be pseudoeffective, as we do have a Boucksom-Zariski decomposition in this case, thanks to Boucksom. This will be relevant in our Corollary \ref{cor:eff-bir-bound}.
\end{remark}

\begin{remark} In fact, the proof provides an effective way to find a bound in any of the two items of Theorem \ref{ZariskiBNC} knowing the bound in the other. More precisely, if the self-intersections are bounded by $b$ then the denominators are bounded by $|b|^{\rho(X)-1}!$. Conversely if the denominators are bounded by $d$ then the self-intersections are bounded by $d!\cdot d \cdot \card(A_{{\rm NS}(X)})$ where $\card(A_{{\rm NS}(X)})$ is the discriminant of the Neron--Severi lattice of $X$.
\end{remark}

In order to find a global bound for denominators it is then sufficient to show that bounded negativity holds in the context of irreducible symplectic manifolds. In dimension 2 it is well known by the adjunction formula that the squares of irreducible curves are bounded from below by $-2$. In higher dimension, we first need to recall results devoted to the divisibility of exceptional prime divisors and their geometry due to Druel and Markman.

\begin{prop}[{\cite[Proposition 1.4 and Remark 4.3]{Druel}}]\label{prop:druel_prime}
Let $X$ be a projective irreducible symplectic manifold  and let $E$ be a prime exceptional divisor on it. 
\begin{enumerate}
\item[(i)] Then there exists another irreducible symplectic manifold $X'$ birational to $X$ and a contraction $\pi': X'\rightarrow Y'$ to a normal variety whose exceptional locus is the strict transform $E'$ of $E$. 
\item[(ii)] If $l$ is a  general primitive exceptional curve  (i.e. a curve whose strict transform is a general fiber of $E'\to \pi'(E')$), then $l$ is either a smooth $\mathbb{P}^1$ or the union of two such $\mathbb{P}^1$ meeting in a point and moreover $E\cdot l=-2$.
\end{enumerate}
\end{prop}
\begin{prop}[{\cite[Corollary 3.6 part 1 and 3]{markman}}]\label{prop:mark_prime}
Let $X,E$ and $l$ be as above. Let us identify $H^2(X,\mathbb{Q})^\vee$ with $H_2(X,\mathbb{Q})$ and let $E^\vee$ be the map given by $q(E,\cdot)$ with the Beauville-Bogomolov-Fujiki form. Then
\begin{enumerate}
\item[(i)] the class of $l$ is $\frac{-2E^\vee}{q(E)}$, and
\item[(ii)] either $E$ is primitive or $E/2$ is.
\end{enumerate}
\end{prop}
From these two results, we deduce the following.
\begin{prop}\label{prop:BNChk}
Let $X$ be a projective irreducible symplectic variety and let $E$ be a prime exceptional divisor on $X$. Assume that the conclusions of Propositions \ref{prop:druel_prime} and \ref{prop:mark_prime} hold.
Then $|q(E)|\leq 4\card(A_X)$, where $A_X$ is the finite discriminant group $H^2(X,\mathbb{Z})^\vee/H^2(X,\mathbb{Z})$. 
\end{prop}
\begin{proof}  
By the conclusion of Proposition \ref{prop:mark_prime}, either $E$ or $E/2$ is primitive. Let $D$ be this primitive class. Let $d$ be the divisibility of $D$, that is the positive generator of the ideal $q(D,H^2(X,\mathbb{Z}))\subset \mathbb{Z}$. Notice that $d$ divides $q(E)$. From the conclusion of Proposition \ref{prop:mark_prime}, item (i), we have that $\frac{-2q(E,\cdot)}{q(E)} $ is an integral class, which means that $q(E)$ divides $2q(E,\cdot)$. This last term generates the ideal $2d\mathbb{Z}\subset \mathbb{Z}$ if $E$ is primitive and $4d\mathbb{Z}\subset \mathbb{Z}$ otherwise. Therefore, we have one of the following:
\begin{itemize}
\item[a)] $D=E$ and $q(E)=kd,\ k=1,2$.
\item[b)] $2D=E$ and $q(E)=kd,\ k=1,2,4$.
\end{itemize}
Hence, we only need to bound $d$. Let us consider the lattice $L:=H^2(X,\mathbb{Z})$ (which is a topological invariant) and the natural inclusion $L\hookrightarrow L^\vee$ given by sending an element $t\in L$ to $q(t,\cdot)$. This inclusion has finite index and $L^\vee/L=A_X$ is a finite group.  Since $\frac{-2q(E,\cdot)}{q(E)} $ is an integral class, by a) or b) above,   the element $\frac{2q(D,\cdot)}{kd}$ lies in $L^\vee$ and it gives a class in $A_X$ of order $kd/2$, as $D$ is primitive. Therefore, $d$ divides $2\card(A_X)$ and our claim holds.
\end{proof}
Notice that, if $\dim(X)\geq 4$, the bound $4\card(A_X)$ is sharp, as it is obtained in the case of the Hilbert-Chow exceptional divisor on the Hilbert scheme of points on a $K3$ surface. We can say the same about generalized Kummer varieties. However, when $A_X$ is not cyclic, its order can be replaced by the highest order of its elements. In the following table we summarize in the four known deformation classes what is $A_X$, the order $d$ of its highest order element and the highest absolute value of the square of a prime exceptional divisor:
\begin{center}
	\begin{tabular}{|c|c|c|c|}
\hline
Deformation type & $A_X$ & Order $d$ & Highest negative divisor square \\
\hline
$K3^{[n]}$-type & $\mathbb{Z}_{/(2n-2)\mathbb{Z}}$ & 2n-2 & 8n-8 \\
\hline
Kummer $n$ type & $\mathbb{Z}_{/(2n+2)\mathbb{Z}}$ & 2n+2 & 8n+8 \\
\hline
O'Grady sixfolds & $\mathbb{Z}_{/2\mathbb{Z}}\times\mathbb{Z}_{/2\mathbb{Z}}$ & 2 & 8\\
\hline
O'Grady tenfolds & $\mathbb{Z}_{/3\mathbb{Z}}$ & 3 & 6 \\
\hline 
\end{tabular}
 \end{center}
The above values can be found in \cite[Section 9]{mark_tor} for $K3^{[n]}$-type manifolds, in \cite{yoshi} for Kummer $n$ type, in \cite{rap_form} for O'Grady's sixfolds and \cite{mark_g2} for O'Grady's tenfolds.

\begin{remark} 
There has been an effort to prove a similar boundedness result for a wider class of divisors, called wall divisors
(see \cite[Definition 1.2]{Mon}). Prime exceptional divisors are automatically wall divisors (see e.g. \cite[Lemma 1.4]{Mon}), while an example of a wall divisor which is not prime exceptional is  obtained by considering a divisor whose class is dual to the class of a line in a projective plane inside a smooth irreducible symplectic fourfold. A complete boundedness result for wall divisors (with no hypothesis on the deformation class of the irreducible symplectic manifold) was proven in \cite{amver_last}. However their result does not give an explicit bound, which is what we need and prove here to obtain an explicit bound for the denominators. The result is rather formulated for MBM classes, but these are  exactly the classes of curves dual to wall divisors (see e.g. \cite[Remark 2.4]{TAMS}).
\end{remark}

\begin{corollary} \label{corofZariskiBNC} Let $X$ be a projective irreducible symplectic variety of dimension $>2$. Assume that the conclusions of Propositions \ref{prop:druel_prime} and \ref{prop:mark_prime} hold for $X$. Then the denominators of the coefficients of the positive and negative parts in the Boucksom-Zariski decompositions of all effective Cartier divisors are  bounded by $(4\card(A_X))^{\rho(X)-1}!$. In particular, on every $Y$ deformation equivalent to $X$ these coefficients are globally bounded by $(4\card(A_X))^{h^{1,1}(X)-1}!$
\end{corollary}
\begin{proof}
The first part follows immediately from the proof of Theorem \ref{ZariskiBNC} and Proposition \ref{prop:BNChk}. For the second statement we use the fact that the lattices $(H^2(X,\mathbb{Z}), q)$ and $(H^2(Y,\mathbb{Z}), q)$ are isometric to each other follows from by Ehresmann's Lemma and \cite{Beau84}.
\end{proof}

\section{Applications to effective birationality}

We collect in this section the applications to effective birationality of big and effective line bundles $L$ on projective holomorphic symplectic manifolds. 

If one could control the singularities of members of the linear system $|L|$, the existence of a uniform (although not explicit) bound would follow from \cite[Theorem 1.3]{HMX}. For the history of the problem of the effectivity of birational pluri(log)canonical maps (and more generally of the Iitaka fibration) in the framework of the MMP and a list of references we refer the interested reader to \cite{HMX}, \cite{BZ} and the very recent \cite{Bi}.
\begin{proof}[Proof of Corollary \ref{cor:eff-bir-bound}]

Let $L$ be a big line bundle on $X$. 
Consider the Boucksom-Zariski decomposition 
$$
a L = P + N,
$$
where $P$ is a $q_{X}$-nef Cartier divisor, $N$ is a Cartier $q_X$-exceptional divisor, and 
$$ 
a  = (4\Card(A_X))^{\rho(X)-1} !
$$
is the integer given by Theorem \ref{thm:bound-intro} and clearing the denominators in the Boucksom-Zariski decompositions.
Since $L$ is big by hypothesis, the positive part $P$ is big. 

Since the cone of $q_X$-nef classes coincides with the closure of the birational K\"ahler cone (see Remark \ref{rem nef = closure bir kahler}), therefore there exists a smooth projective irreducible symplectic variety $X'$ and a birational map
$$
\phi :X\dashrightarrow X'
$$
such that 
$$
P':= f_* P
$$
is an integral, nef and big divisor on $X'$.  
Now we can apply Koll\'ar's extension \cite[Theorem 5.9]{Kol} to nef and big divisors of the Angehrn-Siu result to $P'$, 
to deduce that 
for all $m\geq ({\rm dim} X +2)({\rm dim} X+3)/2$
the morphism associated to $|mP'|$ is injective on $X'$.
As a by-product we then obtain that the linear system $|amL|$
separates two generic points on $X$.
\end{proof}

\begin{cor}\label{cor:bir-bound}
Let $n$ be a positive integer and $C$ be a positive constant.
Then the family of all smooth projective irreducible symplectic varieties of dimension
$2n$, of a fixed deformation type and endowed with a big line bundle of volume at most $C$, is birationally bounded (i.e. there exists an algebraic variety parametrizing birational equivalence classes of such varieties).
\end{cor}
\begin{proof}
By Corollary \ref{cor:eff-bir-bound} and \cite[Lemma 2.2]{HM} it follows that any projective irreducible symplectic varieties of dimension
$2n$, endowed with a big  line bundle $L$ having volume $\leq C$ is birational to a subvariety of a
projective space (which by projection we may assume to be $\mathbb P^{4n+1}$) of degree at most $m_0^{2n}C$, where 
$$
m_0=\frac{1}{2}(2n+2)(2n+3)(4\Card(A_X))^{h^{1,1}(X)-1} !
$$ 
is the integer given by (\ref{eq:eff}). 
Such subvarieties are parametrized by the points of an algebraic variety (the Chow scheme) and the result follows.
\end{proof}

\begin{remark}\label{rmk:ext}
It is interesting to notice that the proof of Corollary \ref{cor:eff-bir-bound}, hence of Corollary \ref{cor:bir-bound}, goes through for possibly singular projective irreducible symplectic varieties $X$ as soon as the Boucksom-Zariski decomposition holds for pseudoeffective divisors, Propositions \ref{prop:druel_prime} and \ref{prop:mark_prime}
hold for $X$ and the cone of $q_X$-nef classes lies inside the closure of the birational K\"ahler cone. The outcome is the same effective birationality result and birational boundedness for such varieties.  {\it Added in the revision process:} This is exactly what has been done in \cite[Theorem 1.1]{LMP3} for primitive symplectic varieties, using results from \cite{LMP2}. 
\end{remark}

\section*{Acknowledgments}
The fourth author would like to thank Ekaterina Amerik for very useful suggestions about Markman's results and for a nice discussion about the topic of the note and to John Lesieutre for discussions about the bounded negativity conjecture. We wish to thank Fran\c cois Charles for useful discussions on the effective birationality. We thank Christian Lehn for several discussions and suggestions about the singular case. We are also grateful to the anonymous referee for a careful reading of the paper and many valuable comments. The first author was supported by the Polish National Science Center project number 2013/10/E/ST1/00688. The third author  was partially supported by the Projet ANR-16-CE40-0008 ``Foliage''.
The fourth author was partially supported by the Foundation for Polish Science Start Scholarship 76/2018.

\end{document}